\documentclass[12pt,reqno]{article}

\usepackage[usenames]{color}
\usepackage{amssymb}
\usepackage{amsmath}
\usepackage{amsthm}
\usepackage{amsfonts}
\usepackage{amscd}
\usepackage{graphicx}

\usepackage[colorlinks=true,
linkcolor=webgreen,
filecolor=webbrown,
citecolor=webgreen]{hyperref}

\definecolor{webgreen}{rgb}{0,.5,0}
\definecolor{webbrown}{rgb}{.6,0,0}

\usepackage{color}
\usepackage{fullpage}
\usepackage[title]{appendix}
\usepackage{float}
\usepackage{graphics}
\usepackage{latexsym}
\usepackage{epsf}
\usepackage{breakurl}
\usepackage{enumerate}
\usepackage{enumitem}
\usepackage{hyperref}
\usepackage{comment}
\usepackage{dsfont,tipa,mathtools,listings, rotating, multirow}

\newcommand{\seqnum}[1]{\href{http://oeis.org/#1}{\underline{#1}}}
\newcommand{\bburl}[1]{\textcolor{blue}{\url{#1}}}

\theoremstyle{plain}
\newtheorem{thm}{Theorem}[section]
\newtheorem{lem}[thm]{Lemma}

\theoremstyle{definition}
\newtheorem{defi}[thm]{Definition}
\newtheorem{algorithm}[thm]{Algorithm}

\theoremstyle{remark}
\newtheorem{rek}[thm]{Remark}

\newcommand\be{\begin{equation}}
\newcommand\ee{\end{equation}}
\newcommand\bea{\begin{eqnarray}}
\newcommand\eea{\end{eqnarray}}
\newcommand\bi{\begin{itemize}}
\newcommand\ei{\end{itemize}}
\newcommand\ben{\begin{enumerate}}
\newcommand\een{\end{enumerate}}

\newcommand{\su}{\sum\limits}
\newcommand{\sui}[1]{\su_{i=#1}^\infty}
\newcommand{\z}{\mathbb{Z}}

\newcommand{\sfr}{\text{\textsc{SF}}}

\newcommand{\rtsf}{\text{\textsc{RTSF}}}

\numberwithin{equation}{section}

\begin{document}

\begin{center}
\vskip 1cm{\LARGE\bf
Walking to Infinity Along Some Number Theory Sequences
}
\vskip 1cm
\large
Steven J.\ Miller\\
Department of Mathematics and Statistics\\
Williams College\\
Williamstown, MA 01267\\
USA \\
\href{mailto:sjm1@williams.edu}{\tt sjm1@williams.edu} \\
\ \\
Fei Peng\\
Department of Mathematics\\
National University of Singapore\\
Singapore 138601\\
Singapore \\
\href{mailto:pfpf@u.nus.edu}{\tt pfpf@u.nus.edu} \\
\ \\
Tudor Popescu\\
Department of Mathematics\\
Brandeis University\\
Waltham, MA 02453\\
USA \\
\href{mailto:tudorpopescu@brandeis.edu}{\tt tudorpopescu@brandeis.edu} \\
\ \\
Joshua M.\ Siktar\\
Department of Mathematics\\
University of Tennessee\\
Knoxville, TN 37916\\
USA \\
\href{mailto:jsiktar@vols.utk.edu}{\tt jsiktar@vols.utk.edu} \\
\ \\
Nawapan Wattanawanichkul\\
Department of Mathematics\\
University of Illinois Urbana-Champaign\\
Urbana, IL 61801\\
USA\\
\href{mailto:nawapan2@illinois.edu}{\tt nawapan2@illinois.edu} \\
\ \\
The Polymath REU Walking to Infinity Group\\
\ \\
\end{center}

\vskip .2 in

\begin{abstract}
An interesting open conjecture asks whether it is possible to
walk to infinity along primes, where at each step we add one more digit somewhere in the prime. We present different greedy models for prime walks to predict the long-time behavior of the trajectories of orbits, one of which has similar behavior to the actual backtracking one. Furthermore, we study the same conjecture for square-free numbers, which is motivated by the fact that they have a strictly positive density, as opposed to the primes. We introduce stochastic models and analyze the walks' expected lengths and the frequencies of the digits added. Lastly, we prove that it is impossible to walk to infinity in other important number-theoretical sequences, such as perfect squares or primes in different bases.
\end{abstract}

\section{Introduction}\label{sec:introduction}

\subsection{Background}\label{sec:background}
This paper is motivated by a simple question: is it possible to walk to infinity along the primes? By this we mean starting with a prime number, appending one digit to it to form a new prime, and repeating endlessly. Note that if each time we are appending to the left an unlimited number of digits, the answer would be positive: one can prove this using Dirichlet's theorem for primes in arithmetic progression: given any prime $p$ other than $2$ and $5$, choose some integer $m$ so that $10^m > p$. As $10^m$ and $p$ are relatively prime, there are infinitely many primes congruent to $p$ (mod $10^m$). Any such prime is obtainable by appending digits to the left of $p$. We can then repeat this process to walk to infinity.

Another natural interpretation is to append one digit at a time, to the right. This greatly reduces the likelihood of an infinite walk. For example, one may start a walk as $3, 31, 317$, and find that $317$ cannot be extended further (by one digit to stay a prime). In fact, starting with any one-digit prime, the longest ``prime walk" (via appending one digit a time to the right) always has length $8$. For example, the optimal walk sequence starting with $3$ is $$\{3, 37, 373, 3733, 37337, 373379, 3733799, 37337999\}.$$

To see this, we introduce the notion of a \textit{right truncatable prime}, which is a prime that remains prime after removing the rightmost digits successively. It is known that there are exactly $83$ right truncatable primes, with the largest one being $73939133$ \cite{TP}. Every right truncatable prime with $d$ digits corresponds to a prime walk of length $d$ starting with a one-digit prime (and vice versa), so the longest such walk has length $8$. 

Without the one-digit starting-point restriction, it is possible to have longer walks: $$\{19, 197, 1979, 19793, 197933, 1979339, 19793393, 197933933, 1979339333\}$$ is a walk with step size 1 and length $9$, while $$\begin{gathered}\{409, 4099, 40993, 409933, 4099339, 40993391, 409933919, 4099339193,\\40993391939, 409933919393, 4099339193933\}\end{gathered}$$ is one of length $11$. In particular, an exhaustive search shows that the above is the longest prime walk with a starting point less than $1{,}000{,} 000$, tied with $$\begin{gathered}\{68041, 680417, 6804173, 68041739, 680417393, 6804173939, 68041739393,\\ 680417393939, 6804173939393, 68041739393933, 680417393939333\}.\end{gathered}$$ On the other hand, some small primes do not have long walks, such as $11$, whose longest walk is $\{11, 113\}$, and $53$, which fails immediately as the numbers between $530$ and $539$ are all composite.

This discussion suggests two central questions.

\begin{itemize}

\item Is it possible to walk to infinity along the primes, where each prime in the sequence is the result of appending one digit to the right of the previous? From the last observation, we cannot do so by starting with a one-digit prime. Some remainder analysis shows that if there is an infinite prime walk (in base $10$), it eventually only appends $3$'s and $9$'s. To clarify, in view of remainders modulo $2$ and $5$, we can never append an even number or a $5$. Then consider remainders modulo $3$. As $1$ and $7$ are both congruent to $1$ (mod $3$), appending them would increase the remainder by 1. Appending $3$ or $9$ would leave the remainder unchanged. As we need to avoid $0$ (mod $3$) at all times, we can append $1$ or $7$ at most once (twice when starting at $3$, but $3$ is already not a promising starting point), and must only use $3$'s and $9$'s afterward.

\item What if, instead of appending just one digit, we append at most a bounded number of digits to the right? More generally, what if the number of digits we append in moving from $p_n$ to $p_{n+1}$ is at most $f(p_n)$ for some function $f$ tending to infinity? Unlike the case of appending to the left, we cannot immediately deduce the answer by appealing to Dirichlet's theorem for primes in arithmetic progressions. 

\end{itemize}


\subsection{Stochastic models}\label{sec:stochasticmodels}

Like most problems in number theory, the aforementioned questions are easy to state but resist progress. We thus consider instead related random problems to try and get a sense of what might be true. Such models have been used elsewhere with great success, from suggesting there are only finitely many Fermat primes to the veracity of the Twin Prime and Goldbach conjectures. 

For example, recall the $n$\textsuperscript{th} Fermat number is $F_n = 2^{2^n} + 1$. The prime number theorem says that the number of primes up to $x$ is about $x/\log x$, and thus one often models a randomly chosen number of order $x$ as being prime with probability $1/\log x$. This is the famous Cram\'er model; while it is known to have some issues \cite[pp.\ 507--514]{MS}, it gives reasonable answers for many problems. If we let $\{Z_n\}$ be independent Bernoulli random variables where $Z_n = 1$ with probability $1/\log F_n$, then the expected number of $Z_n$'s that are 1 (and thus the expected number of Fermat primes) is 
\be \mathbb{E}\left[\sum_{n=0}^\infty Z_n\right] \ = \ \sum_{n=0}^\infty \frac1{\log (2^{2^n} + 1)} \ \approx \ 2.24507722, \ee which is reasonably close to the number of known Fermat primes, five, coming from $n \in \{0,1,2,3,4\}$.

As primes lack much inherent structure, we ask related questions of other sequences, such as square-free numbers. From our heuristic model and numerical explorations, we do not believe one can walk to infinity through the primes by adding a bounded number of digits to the right; however, we believe it is possible for square-free numbers. For example, starting with $2$, we can get a really long walk just by always appending the smallest digit that yields a square-free number. The following sequence only shows the first $17$ numbers obtained using this greedy approach. We do not expect this sequence to terminate soon.
$$\begin{gathered} \{2,\ 21,\ 210,\ 2101,\ 21010,\ 210101,\ 2101010,\ 21010101,\ 210101010,\ 2101010101,\\
 21010101010,\ 210101010101,\ 2101010101010,\ 2101010101010102,\\
210101010101021,\ 2101010101010210,\ 21010101010102101,\ \ldots \}.\end{gathered}$$

While the fraction of numbers at most $x$ that are prime is approximately $1/\log{x}$, which tends to zero, the fraction which are square-free tends to $1/\zeta(2) = 6/\pi^2$, or about $60.79\%$ (for more details, see Section \ref{sec:rmlesfw}). Thus, there are tremendously more square-free numbers available than primes. In particular, once our number is large, it is unlikely that \emph{any} digit can be appended to create another prime. Thus, it should be impossible to walk to infinity among the primes by appending just one digit on the right. However, for square-free numbers, we expect to have several digits that we can append and stay square-free, leading to exponential growth in the number of paths.

Computationally, a bottleneck of investigating prime or square-free walks is the hardness of factorization, which is necessary to determine whether the current number is prime/square-free. To overcome this difficulty, we describe fast, stochastic models that approximate the actual walks.

Explicitly, we consider the following random processes: given a sequence whose last term is $x$, we want to assign an appropriate probability of being able to append an additional digit to the right. We assume each term is independent of the previous, and the probability that a digit can be appended to $x$ is $p(x)$. Thus, the probability will decrease as $x$ increases for primes but is essentially constant for square-free numbers. Furthermore, for prime walks, we present two different models: the first one randomly selects a digit among $1$, $3$, $7$, and $9$ and appends it to the number, while the second (refined) random model first checks what digits yield a prime number in the next step and then randomly selects one from the set. We assume all numbers with the same number of digits are equally likely to be in the sequence for simplicity. For the primes base 10, we cannot append a digit that is even or a $5$, whereas, for square-free numbers, we cannot append a digit such that the sum of the digits is $9$ or the last two digits are a multiple of $4$. One could consider more involved models taking these into account.

We approximate that if a number has $k$ digits, the number of primes of $k$ digits in base~$b$ is \be\label{eq:BaseBPrimeCount}\frac{b^k}{\log{b^k}} - \frac{b^{k-1}}{\log{b^{k-1}}} \ = \ \frac{b^{k-1}}{\log{b}} \cdot \left(\frac{b}{k} - \frac{1}{k-1} \right) \ = \ \frac{b^{k-1}((k-1)b - k)}{k(k-1)\log{b}} \approx \frac{(b-1) \cdot b^{k-1}}{k\log{b}}.\ee Since there are $(b-1)b^{k-1}$ numbers with exactly $k$ digits in base $b$, we assume the probability that a $k-$digit number is prime is $1/{(k\log{b})}$, and assume that the events of two distinct numbers being prime are independent.

Our main focus is the expected value and distribution of lengths of walks under these stochastic settings. Such probabilistic models have had remarkable success in modeling other problems, such as the $3x+1$ map and its generalizations \cite{KL}. They also have several issues. In particular, we assume that the numbers formed by appending the digits under consideration are all independent in our desired sequence. However, this yields a simple model with easily computed results on how long we expect to be able to walk in the various sequences from different starting points.

In the rest of the paper, the \textit{expected walk length} naturally refers to the expected value of the length of the walk. Typically, there is either an explicit, finite collection (or sample) of walks to empirically determine the expected length from, or a corresponding model for a walk, under which the expected walk length can be computed theoretically from the model's definition.


\subsection{Results}\label{sec:results}

We compare the random model with observations of the actual sequences. We present the two random models for prime walks and show that the refined one is very close, in some sense, to the actual sequence. In particular, when considering prime walks with starting point less than a million, the difference between the experimental expected length for the careful greedy model and the real expected value for the primes is $0.14$, less than $7$ percent of the expected length of the real prime walks of $2.07$. 

Furthermore, we note that the model becomes more precise as the starting point increases, and the prime numbers become more sparse. As the starting point increases, the number of primes from which we randomly choose to continue decreases. Then, we also look at the frequency of the digits added at each step and see that the refined model approximates the real world extremely well. Lastly, while we discuss infinite prime walks, we extend our predictions for the case where we are allowed to insert a digit anywhere, rather than only to the right.

On the other hand, when investigating square-free walks, we present the experimental expected length of our random models. Furthermore, we remark on the discrepancies in the frequencies of added digits, and give the number-theoretic reasons for these discrepancies.

Although we use stochastic models for prime and square-free walks, there are some sequences and restrictive scenarios for which we can prove several results regarding walks to infinity, for example, prime walks in base $2, 3, 4, 5$, and $6$, and on perfect squares. 

A related problem is the Gaussian Moat problem, which asks whether it is possible to walk to infinity on Gaussian primes with steps of bounded length. Extensive research has been done on this. For example, Gethner, Wagon, and Wick in \cite{GWW} and Loh in \cite{Loh} proved numerous results related to the problem. Some of the authors of this article examined the behavior of prime walks in different number fields \cite{Li} and proved that it is impossible to walk to infinity on primes in $\mathbb{Z}[\sqrt{2}]$ if the walk remains within some bounded distance from the asymptotes $y = \pm x/\sqrt{2}$.

The main results of the current study are summarized as follows.

\vspace{0.3 cm}

\textit{Prime walks}
\begin{itemize}
    \item Expressions for the expected prime walk lengths under different models are given by equations \eqref{eq:theoreticalModel1Simplified}, \eqref{eq:EV-refined}.
    
    \item Comparisons of the two prime walk models and the actual primes can be found in Tables \ref{table:primecomparison}, \ref{table:appending1379<100000}, \ref{table:appending1379}, and \ref{table:appending1379>100000}.
    
    \item A proof that it is impossible to walk to infinity on primes in base $2$ by appending no more than $2$ digits is given in Theorem \ref{prime-base-2}, while Lemma \ref{lem:base3To6} shows that it is impossible to walk to infinity on primes in bases $3, 4, 5$, or $6$ by appending one digit to the right.
\end{itemize}

\textit{Square-free walks}
\begin{itemize}
    \item The expected lengths of square-free walks given by our models are presented in Tables \ref{table:square-freecomparison} and \ref{table:greedysquare-freecomparison}, while Theorem \ref{thm:highprobSF} shows that there exists an infinite random square-free walk from most starting points.
    
    \item Table \ref{table:square-freefreq} presents the frequencies of the digits added in square-free walks, and Remark \ref{sq-free-observations} explains why some digits appear more often than others.
    
    \item Theorems \ref{E2estimate} and \ref{E10estimate} give tight bounds on the theoretical expected lengths of square-free walks in bases $2$ and $10$, respectively.
    \end{itemize}


\section{Initial models of prime walks}\label{sec:modelingPrimeWalks}

\subsection{Models}\label{sec:rmlegpw}
We now estimate the length of these random walks in base $b$, so there are $b$ digits we can append. If our number has $k$ digits, then from \S\ref{sec:stochasticmodels}, the probability a digit yields a successful appending is approximately $1/{(k\log{b})}$, as we are assuming all possible numbers are equally likely to be prime. For example, if $b=10$, we are not removing even numbers or 5 or numbers that make the sum a multiple of 3. Thus, the probability that at least one of the $b$ digits works is $1$ minus the probability they all fail, or 

\be \label{ModelAEq}
1 - \left(1 - \frac1{k\log{b}}\right)^b. 
\ee

The first stochastic model for prime walks can be described as follows. 

\begin{algorithm}[Greedy Prime Walk in Base $b$]\label{GreedyPrimeWalkDesignation}
Each possible appended number is independently declared to be a random prime with probability as described by the reasoning used to deduce \eqref{ModelAEq}. Choose one of the admissible digits uniformly at random and check if the obtained number is prime; if it is not, stop and record the length; otherwise, continue the process. 
\end{algorithm}

 This algorithm can be imagined as a greedy prime walk, as we are not looking further down the line to see which of many possible random primes would be best to choose to get the longest walk possible. We call this the greedy model. Furthermore, note that we may easily improve the model in base $10$ by appending from $\{1, 3, 7, 9\}$. We discuss this improvement later in this section, and compare it to this initial greedy model. 

In order to compute the theoretical expected length of such a walk, starting at a one-digit random prime in base $b$, we count the probabilities in two different ways; note that the expected length is just the infinite sum of the probabilities that we stop at the $n$\textsuperscript{th} step times $n$. For brevity, let $A_n$ denote the event that the walk has length at least $n$, and $B_n$ denote the event that the walk has length exactly $n$. Since it is clear that the collection of events $\{B_i\}^{\infty}_{i = 1}$ is pairwise independent and that $A_n = \cup_{i=n}^{\infty} B_i$, it follows that
\be\label{eq:ABexpansionprimes} \sum_{n=1}^\infty \mathbb{P}[\text{walk has length at least }n] \ =\ \sum_{n=1}^{\infty}n\mathbb{P}[\text{walk has length exactly }n].\ee
Note that the sum on the right hand side of \eqref{eq:ABexpansionprimes} is the expected walk length in our greedy model, while the sum on the left equals
\be\label{eq:expectedlengthsimple}
    \sum_{n=0}^{\infty}\prod_{k = 1}^{n-1} \left(1 - \left(1 - \frac{1}{k\log{b}}\right)^{b}\right),
\ee
where each term in the sum represents the probability that there is a random prime with which we can extend the walk for the first $n - 1$ steps, without considering the $n^{\text{th}}$ step.
In particular, the expected length in base $10$ when starting with a single digit is $4.690852$, however we are also interested in other bases. With that in mind, we denote by $Y_{s, b}$ the random variable indicating the length of a prime walk with a chosen prime starting point with at most $s$ digits in base $b$.  Furthermore, by multiplying by the approximate number of primes with exactly $r$ digits and dividing by the approximate number of primes with at most $s$ digits, we get that the theoretical expected length of a walk with a starting point at most $s$ digits in base $b$ is
\begin{equation}\label{eq:theoreticalModel1}
E[Y_{s, b}] \ = \ \frac{1}{\frac{b^s}{s\log{b}}}\left(\sum_{r=1}^{s}\frac{(b-1)b^{r-1}}{r\log{b}}\left(\sum_{n=0}^{\infty}\prod_{k = r}^{n-1} \left(1 - \left(1 - \frac{1}{k\log{b}}\right)^{b}\right)\right)\right),
\end{equation}
which simplifies to
\begin{equation}\label{eq:theoreticalModel1Simplified}
E[Y_{s, b}] \ = \ \frac{s(b-1)}{b^s}\left(\sum_{r=1}^{s}\frac{b^{r-1}}{r}\left(\sum_{n=0}^{\infty}\prod_{k = r}^{n-1} \left(1 - \left(1 - \frac{1}{k\log{b}}\right)^{b}\right)\right)\right).
\end{equation}
Table \ref{table:numdigitsstartingpoint} contains the expected lengths as we vary the starting point and base; one can view this model as a greedy random prime walk because we always take another step if possible, with no regard to how many steps we may be able to take afterward; thus, all decisions are ``local."

Note that the theoretical expected length of the walk in base $10$ starting with a one-digit number, $4.22$, is different than the one we computed earlier, $4.69$. This is because we multiplied $4.69$ by the approximation factor ${(b-1)}/{b}$; i.e., $0.9$. More importantly, note that in base $10$ we can only append $\{1, 3, 7, 9\}$ and hope to stay prime since primes greater than $5$ are odd and not divisible by $5$.

This suggests a simple improvement to the model base $10$: \emph{we only allow the four digits $1$, $3$, $7$, and $9$ to be appended on the right.} Henceforth, we will only use this improved version. To do this, we have to make a couple of changes in the formula \eqref{eq:theoreticalModel1Simplified}: first, we replace the numerator of ${1}/{(k \log b)}$ with $10/4$, as we have the same number of primes despite having less freedom; furthermore, instead of raising $1 - 10/{(4k \log b)}$ to the $b^{\text{th}}$ power (in this case, $10$), we raise it to the fourth power as only 4 options are left. We shall call this the \emph{greedy model}. The expected walk length under this model is presented in Table \ref{table:numdigitsstartingpoint} as 10'. 

Denote by $\widetilde{Y_{s, b}}$ the random variable denoting the length of walking according to this modified algorithm. By modifying our earlier analysis, we obtain the formula for the expected length of this model in base $b$ to be
\be \label{eq:EV-refined} E[\widetilde{Y_{s, b}}] \ = \ \frac{s(b-1)}{b^s}\left(\sum_{r=1}^{s}\frac{b^{r-1}}{r}\left(\sum_{n=0}^{\infty}\prod_{k = r}^{n-1} \left(1 - \left(1 - \frac{b}{\phi(b)k\log{b}}\right)^{\phi(b)}\right)\right)\right).\ee
where $\phi(n)$ denotes Euler's totient function.
 
\renewcommand{\arraystretch}{1.5}
\begin{table}[ht]
\centering
\begin{tabular}{p{5pt}c|ccccccc}
\multicolumn{2}{}{} & \multicolumn{7}{c}{Maximum number of digits of starting point} \\
     & & 1    &  2   & 3    & 4  & 5  & 6 & 7   \\ \cline{2-9}
\multirow{9}{*}{\rotatebox[origin=c]{90}{Base}}
&  2  & 5.20 &  9.90 &  11.62 & 11.45 & 10.40 & 9.08 &  7.79 \\
 & 3  & 5.05 & 7.75 &  7.60 & 6.53 & 5.40 & 4.49 & 3.80 \\
  & 4 & 4.87 & 6.55 &  5.86 & 4.79 & 3.92 & 3.29 & 2.85 \\
  & 5 & 4.71 & 5.79 &  4.92 & 3.96 & 3.25 & 2.78 & 2.45 \\
  & 6 & 4.57 & 5.27 &  4.34 & 3.48 & 2.89 & 2.49 & 2.22 \\
  & 7 & 4.46 & 4.89 &  3.95 & 3.17 & 2.65 & 2.31 & 2.08 \\
  & 8 & 4.37 & 4.59 &  3.67 & 2.95 & 2.49 & 2.19 & 1.98 \\
  & 9 & 4.29 & 4.36 &  3.45 & 2.79 & 2.37 & 2.09 & 1.91 \\
  & 10 & 4.22 & 4.17 & 3.28 & 2.66 & 2.28 & 2.20 & 1.85
  \\
  & 10' & 4.54 & 4.55 & 3.55 & 2.83 & 2.38 & 2.09 & 1.90
\end{tabular}
\caption{\label{table:numdigitsstartingpoint} Expected length of prime walks given by our formula \eqref{eq:theoreticalModel1}, 10' is the blind limited model; this latter model is described in more detail in Section \ref{subsec:refinedGreedy}.}
\end{table}
Our second model is the \emph{careful greedy model}. 

\begin{algorithm}[Careful greedy algorithm]\label{alg:refinedGreedy}
    At each step, we check whether appending any of $1, 3, 7$, or $9$ to the right yields a prime. If there are multiple digits that yield primes, the model selects one of them uniformly at random, and continues the process.
\end{algorithm}

This is a more refined version of Algorithm \ref{GreedyPrimeWalkDesignation}. Indeed, we first check whether any of the numbers obtained after appending an admissable digit is prime; if there are multiple, select one at random, and if there are none stop the process. Obviously, this algorithm approximates the real-world data better than Algorithm \ref{GreedyPrimeWalkDesignation}, but this comes at a computational cost, as at each step we have to check whether $4$ numbers are prime instead of just $1.$

 Lastly, in the \emph{primes} model, we use backtracking to find the longest walk starting at a prime.

While Table \ref{table:numdigitsstartingpoint} presents the expected length of prime walks given by formulas \eqref{eq:theoreticalModel1Simplified} and \eqref{eq:EV-refined}, Tables \ref{table:primecomparison}, \ref{table:appending1379<100000}, \ref{table:appending1379}, \ref{table:appending1379>100000}, and \ref{table:primecomparison2mod3} show the data obtained by computer simulations on our previously described models. We note that this latter table records walks based on the exact number of digits, rather than the maximum number, at the start.


\subsection{Results and comparison of models}\label{subsec:resultsAndComparison}

According to the random probabilistic model of prime walks in \S\ref{sec:rmlegpw}, the expected length of a greedy prime walk, starting with a single digit prime in base 10, is 4.69. We compare this heuristic estimate with the primes.

We present the results of our computer simulations for our blind unlimited and careful limited models in Tables \ref{table:primecomparison}, \ref{table:appending1379<100000}, \ref{table:appending1379}, and \ref{table:appending1379>100000}. The careful greedy model is rather close to the real data whereas the greedy one still predicts some behaviors of the walks. The data for the actual primes is computed by the program that exhaustively searches for the longest prime walk given a starting point. First, let us observe how the number of digits of the starting point affects the expected walk length of the models in Table \ref{table:primecomparison}: it shows that the expected length of the walks decreases as the starting point increases both theoretically and in our random model.

\begin{table}[h!]
\centering
 \begin{tabular}{||c c c c c c c c||}
 \hline
 Start has $r$ digits & $0$ & $1$ & $2$ & $3$ & $4$ & $5$ & $ 6$\\
 \hline\hline
 Blind unlimited model & 1.00 & 1.86 & 1.60 & 1.41 & 1.31 & 1.25 & 1.21 \\
 \hline
 Careful limited model & 4.77 & 5.01 & 3.41 & 2.79 & 2.38 & 2.09 & 1.88\\
 \hline
 primes & 8.00 & 8.00 & 4.71 & 3.48 & 2.71 & 2.28 & 2.03 \\
 \hline
\end{tabular}
\caption{\label{table:primecomparison} Comparison of the expected walk lengths}
\end{table}

Furthermore, we analyze the frequency of digits added in the prime walks under base 10, both for the actual primes and in our models. We originally hypothesized that 3 and 9 appear more often than 1 and 7. This is because 1 and 7 cannot be appended when we start with a prime that is $2$ (mod $3$). We present the frequency of digits in Table \ref{table:appending1379} when the starting point is less than $1,000,000$. As expected, in both our models and the real prime walks, the number of appended 3's is very close to the number of appended 9's while the number of appended 1's is very close to the number of appended 7's. One surprising observations is that there are significantly more 7's in the random models than in the real prime walks. We observe how the starting point affects the frequency of the digits added in Tables \ref{table:appending1379<100000}, \ref{table:appending1379}, and \ref{table:appending1379>100000}.

As mentioned above, we observe that the number of appended $3$'s and $9$'s is larger than the number of appended $1$'s and $7$'s. This is due to the fact that by modulo $3$ considerations, we can only append $3$ or $9$ to a number $2$ (mod $3$) to keep it a prime. In particular, when the starting number is $2\pmod3$, we must append $3$'s and $9$'s, and when it is $1\pmod3$, we can append $1$ or $7$ at most once, and every other digit appended must be $3$ or $9$.\footnote{Since appending 2, 5 or 8 is forbidden for being divisible by 2 or 5, appending a digit either preserves the remainder modulo 3 (the case when appending 3 or 9), or increments it by~1 (the case when appending~1 or~7). In a prime walk, the remainder must never be zero, hence leaving at most one chance of appending~1 or~7 (combined) when the starting number is $1\pmod3$, and no chance at all when it is $2\pmod3$.} We present our model when starting with $2$ (mod $3$) in the following section. Furthermore, this bias will be seen in our models: indeed, if a number is composite after appending a digit, the digit will not be counted for. As the probability of a number being prime immediately after appending $3$ or $9$ is higher than that of being prime immediately after appending $1$ or $7$, the frequency of $3$'s and $9$'s will be higher than that of $1$'s and $7$'s, as can be seen in Tables \ref{table:appending1379<100000}, \ref{table:appending1379}, and \ref{table:appending1379>100000}.

\begin{table}[H]
\centering
 \begin{tabular}{||c c c c c||}
 \hline
 Number appended & 1's & 3's & 7's & 9's \\
 \hline\hline
 Blind unlimited model & 15.6\% & 33.0\% & 19.9\% & 31.3\% \\
 \hline
 Careful limited model & 11.8\% & 36.7\% & 14.2\% & 37.1\% \\
 \hline
 Primes & 12.1\% & 40.2\% & 11.1\% & 36.5\% \\ 
 \hline
\end{tabular}
\caption{\label{table:appending1379<100000} Frequency of added digits in prime walks with starting point less than $100,000$.}
\end{table}

\begin{table}[H]
\centering
 \begin{tabular}{||c c c c c||}
 \hline
 Number appended & 1's & 3's & 7's & 9's \\
 \hline\hline
 Blind unlimited & 15.4\% & 32.7\% & 18.5\% & 33.2\% \\
 \hline
 Careful limited model & 12.5\% & 35.9\% & 14.7\% & 36.8\% \\
 \hline
 Primes & 13.1\% & 38.8\% & 12.2\% & 35.6\% \\
 \hline
\end{tabular}
\caption{\label{table:appending1379} Frequency of added digits in prime walks with starting point less than $1,000,000$.}
\end{table}

\begin{table}[H]
\centering
 \begin{tabular}{||c c c c c||}
 \hline
 Number appended & 1's & 3's & 7's & 9's \\
 \hline\hline
Blind unlimited model & 16.3\% & 32.3\% & 18.5\% & 32.8\% \\
 \hline
 Careful limited model & 12.7\% & 35.8\% & 14.8\% & 36.4\% \\
 \hline
 Primes & 13.3\% & 38.6\% & 12.4\% & 35.5\% \\
 \hline
\end{tabular}
\caption{\label{table:appending1379>100000} Frequency of added digits in prime walks with starting point greater than $100,000$ but less than $1,000,000$.}
\end{table}

We defer a more in-depth discussion of the case where our starting number is $2$ (mod $3$) to Appendix \ref{2mod3}. In the meantime, we use the stochastic prime walks presented thus far to motivate the results of the next subsection, which give conditions on which prime walks are impossible.



\subsection{Main results for prime walks}\label{primeWalkProof}

As mentioned in the introduction, it is possible to walk to infinity on primes by appending an unbounded number of digits to the left at each step. We now show that this statement's counterpart is also true, namely that it is possible to walk to infinity on primes by appending an unbounded number of digits to the right.

\begin{thm}\label{rightunbounded}
    Let $p_0$ be a prime. Then there exists a sequence of infinitely many primes $p_0, p_1, \ldots$ such that for all $i\ge1$, $p_i$ is equal to $10^{n_i}\cdot p_{i-1}+k_i$, for positive integers $n_i$ and $k_i$ with $k_i<10^{n_i}$. 
\end{thm}

\begin{proof} We can restate our goal as follows: given a  fixed prime $p$, we must show that there exists an $n$ such that there is a prime in the interval $[p10^n, (p+1)10^n)$. To do so, we note that for a given $p$ and any fixed $r \in [0, 1]$, there exists an $n$ such that 
\begin{equation}\label{eqdefrandn}
p\ < \ 10^{\frac{1-r}{r}n}-1.
\end{equation} Moreover, given such an $n$, it is then possible to find $x > 0$ such that 
\begin{equation}\label{eqdefxminusxr}
p10^n \ = \ x-x^{r}.
\end{equation} Then, using first \eqref{eqdefrandn} and then \eqref{eqdefxminusxr}, we have that 
\begin{eqnarray*}
    p10^n &<& 10^{\frac{n}{r}} - 10^n \\
    x-x^{r} &<& 10^{\frac{n}{r}} - 10^n.
\end{eqnarray*}
The second inequality implies that $x^{r}<10^n$, for when $x^{r}=10^n$, then $x-x^{r} = 10^{\frac{n}{r}} - 10^n$, and moreover, $x-x^{r}$ is strictly increasing (once it is positive).

Given that $x^{r}<10^n$, then $x-x^{r}>x-10^n$. This means that $p10^n>x-10^n$, and so 
\begin{equation}\label{eqboundonx}
    x \ < \ (p+1)10^n.
\end{equation} All that remains is finding an $r$ such that there is always a prime in the interval $[x-x^r,x]$. Results of this nature are plentiful; most recently, Baker, Harman, and Pintz showed that a value of $r = 0.525$ suffices for $x$ greater than some lower bound $x_0$ \cite{BaHaPi}. Since there exists a prime in the interval $[x-x^{0.525}, x]$ for $x > x_0$, then by our previous definitions there must be a prime contained in the interval $[p10^n,(p+1)10^n)$. Note that in order to guarantee $x>x_0$, it is necessary to choose an $n$ such that $n>\log_{10}((x_0-x_0^r)/p)$ (and such that \eqref{eqdefrandn} holds as well).

That there is a prime in $[p10^n, (p+1)10^n)$ implies that there exist $n$ and $k$ such that $p10^n+k$ is prime, with $k<10^n$. This gives the next prime in our sequence, which thus goes on infinitely. \end{proof}

Now we define the extended Cunningham Chain, which ultimately serves as a tool for proving that one cannot walk to infinity on primes in base $2$ when appending up to only $2$ digits at a time to the right.

\begin{defi} An \textit{extended Cunningham chain} is the infinite sequence $e_1,e_2,\ldots$, generated by an initial prime $p$ and the relation $e_k=2e_{k-1}+1$ (for $k\ge1$ and $e_0=p$). In other words, we have that 
\begin{eqnarray*}
    e_0&=&p, \\
    e_1 &=& 2p+1, \\ 
    e_2 &=& 4p+3, \\
    &\vdots& \\
    e_i &=& 2^ip+2^i-1, \\
    &\vdots&
\end{eqnarray*}
\end{defi}

We show that extended Cunningham chains, no matter their initial prime $p$, contain a sequence of consecutive composite $e_i$'s of arbitrarily length. To do so, we begin with the following lemma.

\begin{lem}\label{power2}
If $k\ge\lceil{\log_2(p+1)\rceil}+2$, then $2^k-(p+1)$ is not a power of $2$.
\end{lem}

\begin{proof} Suppose that there exist $k$ and $n$ such that $2^k-(p+1)=2^n$. Then it is the case that $2^k-2^n=p+1$. Moreover, we have that $2^k-2^n\ge2^k-2^{k-1}=2^{k-1}$. We can thus find a solution for $n$ only if $k<\lceil{\log_2(p+1)\rceil}+2$, for if we take $k\ge\lceil{\log_2(p+1)\rceil}+2$, then $2^{k-1}\ge2^{\lceil\log_2(p+1)\rceil+1}>{p+1}$. We thus have that $p+1=2^k-2^n\ge{2^{k-1}}>{p+1}$, which is a contradiction.
\end{proof}

The next result illustrates the power of Cunningham Chains, which in turn is used to prove the main result of this section.

\begin{thm}\label{longcchain}In any extended Cunningham chain $\{e_k\}^{\infty}_{k = 1}$, given any $n\in\mathbb{Z}_{+}$, there exists $i\in\mathbb{Z}_+$ such that $\{e_i, e_{i+1}, \dots, e_{i+n-1}\}$ are composite.
\end{thm}
\begin{proof}
Set $k=\lceil{\log_2(p+1)\rceil}+2$, and let us consider $i=c\cdot\phi(2^k-p-1)\cdot\phi(2^{k+1}-p-1)\cdots\phi(2^{k+n-1}-p-1)$, where $c\in\mathbb{Z}_{+}$ is arbitrary. Moreover, for each of $2^{k+j}-p-1$ (with $0\le{j}\le{n-1}$), take an odd positive divisor $d_j\mid 2^{k+j}-p-1$ that is greater than 1. We can find such $d_j$ because we have chosen $k$ via Lemma~\ref{power2} such that none of $2^{k+j}-p-1$ are powers of $2$. Because $p+1\equiv2^{k+j}$ (mod ${2^{k+j}-p-1}$), it is also the case that $p+1\equiv2^{k+j}$ (mod ${d_j}$). Thus we have that 
\begin{equation}
e_{i-(k+j)}\ =\ 2^{i-(k+j)}(p+1)-1\ \equiv\ 2^{i-(k+j)}2^{k+j}-1\equiv2^i-1\pmod{d_j}. \\
\end{equation}
However, as 2 is coprime with $d_j$, Euler's theorem gives $2^{\phi(d_j)}\equiv1$ (mod ${d_j}$). Moreover, it is the case that $\phi(d_j)\mid\phi(2^{k+j}-p-1)$, since $d_j\mid(2^{k+j}-p-1)$. Hence we have that
\begin{eqnarray}
    e_{i-(k+j)}&\equiv&2^i-1 \ \equiv \ 2^{c\cdot\phi(2^k-p-1)\cdot\phi(2^{k+1}-p-1)\cdots\phi(2^{k+n-1}-p-1)} - 1 \nonumber\\
    &\equiv&(2^{\phi(d_j)})^{K_j}-1 \ \equiv \ 0\pmod{d_j},
\end{eqnarray}
such that $K_j$ is an integer ($K_j=c\cdot[\phi(2^k-p-1)\cdots\phi(2^{k+n-1}-p-1)]/\phi(d_j)$). \\

We have thus shown that $\{e_{i-k}, e_{i-k-1}, \ldots, e_{i-(k+n-1)}\}$ are composite. Notice that with $c$ sufficiently large, $i$ can be made greater than $k+n-1$,
allowing us to find a sub-sequence of $n$ composite elements for any $n$. Renaming the indices gives the desired result.
\end{proof}
Using this result we can prove that appending $2$ digits at a time to the right is insufficient to walk to infinity on primes in base $2$.

\begin{thm} \label{prime-base-2}
It is impossible to walk to infinity on primes in base $2$ by appending no more than $2$ digits at a time to the right. 
\end{thm} 

\begin{proof}
Since we are appending at most $2$ digits in base $2$, the allowed blocks are $0_2$, $1_2$, $00_2$, $01_2$, $10_2$, and $11_2$. Avoiding even numbers, we are left with $1_2$, $01_2$ and $11_2$.

Appending $01_2$ to a prime $p$ gives $4p+1$. If $p \equiv 2$ (mod $3$), then $4p+1$ is divisible by $3$ and is thus not prime. Moreover, given $p\equiv2$ (mod $3$), then $2p+1$ and $4p+3$ are equivalent to $2$ (mod $3$) as well. Thus if $p \equiv 2$ (mod $3$), we can walk to infinity from that point onward only by appending $1_2$ or $11_2$.

If $p \equiv 1$ (mod $3$), then $4p+1 \equiv 2$ (mod $3$). This brings us to the above case, now applied to $4p+1$. No matter the value of our initial prime $p$, we can therefore append $01_2$ at most once in our walk to infinity. It is thus sufficient to consider the point at which we append only $1_2$ or $11_2$ to eternity. We can then apply Theorem~\ref{longcchain} with $n=2$. Namely, continuously appending $1$ to a prime in  base $2$ creates a generalized Cunningham chain, which we know contains prime gaps of size $2$; hence there will be some point in the prime sequence for which $2p+1$ and $4p+3$ are both composite, and we can walk no further.
\end{proof}

Applying the ideas of the above results allows us to make observations in bases $3, 4, 5$, and $6$; the analogous results and their proofs can be found in Section \ref{subsec:impossibility}. The Section \ref{subsec: Mersenne} also discusses the inability to create walks on the (very sparse) Mersenne primes.

In conclusion, the use of stochastic models suggests that there is no infinite prime walk given by adding one digit at a time to the right, but that it is likely to have one if we can add a digit anywhere.


\section{Modeling square-free walks}\label{sec:sqfreeModel}

\subsection{Model}\label{sec:rmlesfw}
We now turn our attention to square-free walks whose density is positive, in contrast to sequences that have density zero like primes. Due to this difference, we expect that we can construct a square-free walk to infinity. Before discussing the details, we precisely define the square-free integers.

\begin{defi}[Square-Free Integer]\label{def:SqFree}
A \textit{square-free integer} is an integer that is not divisible by any perfect square other than $1$.
\end{defi}

To verify our conjecture, we must be judicious in how we append digits. For example, $231546210170694222$ is a square-free number such that successively removing the rightmost digit always yields a square-free number, but appending any digit to the right yields a non-square-free one. If $Q(x)$ denotes the number of square-free positive integers less than or equal to $x$, it is well-known \cite{MT-B} that \be Q(x)\ \approx \ x \underset{p\text{ prime}}{\prod} \left(1 - \frac{1}{p^2}\right)\ =\ x \underset{p\text{ prime}}{\prod} \frac{1}{1 + \frac{1}{p^2} + \frac{1}{p^4} + \cdots}\ =\ \frac{x}{\zeta(2)}\ =\ \frac{6x}{\pi^2}. \label{eq:density}\ee
In this setting, each possible appended number is independently declared to be a square-free number with probability $p = 6/\pi^2$. We now present our first model for estimating the length of square-free walks. 

\begin{algorithm}[Blind Unlimited Square-Free Walk]\label{sqFreeWalkDesignation}
 Choose one digit uniformly at random from the set $\{0, 1, \dots, 9\}$ and append it: if the obtained number is not square-free, stop and record the length; otherwise, continue the process.
\end{algorithm}

 We present the experimental expected lengths under this model, starting with different number of digits, in Table \ref{table:square-freecomparison}. Akin to Table \ref{table:primecomparison2mod3} we opt to provide data based on the exact number of digits in the starting number, rather than the maximal possible number.
\begin{table}[!htbp]
\centering
 \begin{tabular}{||c c c c c c c c||}
 \hline
 Start has exactly $r$ digits & $0$ & $1$ & $2$ & $3$ & $4$ & $5$ & $ 6$\\
 \hline\hline
 Blind unlimited square-free walk & 1.68 & 2.79 & 2.76 & 2.72 & 2.71 & 2.71 & 2.71 \\
 \hline
\end{tabular}
\caption{\label{table:square-freecomparison} Experimental expected lengths of the square-free walks in base $10$.}
\end{table}

To gain more understanding of the behavior of square-free walks, we first find the probability that the square-free walk is of length exactly $k$, where the number of starting digits is fixed.

\begin{lem}\label{XLem}
    Let $X_m$ denote the length of our random square-free walk, starting with exactly $m$ digits. Then the theoretical expected value of $X_m$ is
    \begin{eqnarray}\label{XLemEq}
        \begin{aligned}
            \mathbb{E}[X_m] \ &\approx \ \begin{cases}1.55&{m=0}\\2.55&{m\ge1}.\end{cases}
        \end{aligned}
    \end{eqnarray}
\end{lem}

\begin{proof}
This proof is standard via the technique of differentiating identities, but we present it anyway for the sake of completeness.
Note that, for all $k\ge 0$ and positive~$m$, $X_m=k$ corresponds to $k-1$ successful appendings followed by an unsuccessful one. However, $X_0=k$ corresponds to $k$ successful ones and 1 unsuccessful one. Thus,
we have that\be\Pr[X_0 = k]\ =\ p^{k}(1-p)\ =\ \Pr[X_m = k+1],\ee
and so $X$ is a geometric random variable. Using the fact that \be \sum_{k=0}^{\infty} p^k\ =\ \frac{1}{1-p} \ee and differentiating term by term (which is permissible due to absolute convergence), we obtain \be\sum_{k=0}^{\infty} kp^{k-1}\ =\ \frac{1}{(1-p)^2}\ \  \Rightarrow\ \ \sum_{k=0}^{\infty} kp^{k}(1-p) \ =\ \frac{p}{1-p}.\ee
Therefore, we find 
\begin{align}\mathbb{E}[X_0] \ &= \ \sum_{k=0}^{\infty}k\Pr[X_0 = k] \ = \ \sum_{k=0}^{\infty} kp^{k}(1-p)\ =\ \frac{p}{1-p}\ =\ \frac{6}{\pi^2 - 6} \ \approx\ 1.55. \label{eq:expected-sq}\\
\mathbb{E}[X_m]\ &= \ 1+\sum_{k=0}^\infty k\Pr[X_m-1=k]\ =\ 1+\sum_{k=0}^{\infty} kp^{k}(1-p)\ \approx\ 2.55.~~~(m\ge1)\end{align}
\end{proof}

Let us now compute the probability that the longest walk starting with a given square-free number is at most $k$. Let $P_{i}$ be the probability that the longest square-free walk has length at most $i$. In particular, $P_{1}$ is the probability that the longest square-free walk has a length of exactly one, i.e., the walk is the starting point. In other words, appending any digit yields a non-square-free number, so \be\label{eq:P1square} P_{1} \ = \ (1-p)^{10} \ = \ \frac{(\pi^2 - 6)^{10}}{\pi^{20}} \ \approx \ 8.58357\times10^{-5}.\ee 

Now, consider the probability that the longest square-free walk has length at most $2$; 
indeed, there are $10$ possible cases where exactly $i$ digits work in the first appending, i.e., $0\le i \le 9$. Then, by using \eqref{eq:P1square} and the Binomial Theorem we have that
\begin{eqnarray} \label{eq:P2square}
P_{2} &=& P_{1} + {10 \choose 1}(1-p)^9pP_{1} + {10 \choose 2}(1-p)^8(pP_{1})^2 + \dots + {10 \choose 10}p^{10}P_{1}^{10} \nonumber\\
&=& {10 \choose 0}(1-p)^{10} + {10 \choose 1}(1-p)^9pP_{1} +  \dots + {10 \choose 10}p^{10}P_{1}^{10} \nonumber\\
&=& (1-p+pP_1)^{10} \ \approx \ 8.59501 \times 10^{-5} .\end{eqnarray}

To compute $P_3$, let $i$ denote the number of digits that we can append in the first step while remaining square-free. Then, there are $10i$ possible numbers after the second appendage. Like $P_2$, we consider cases when there are exactly $0 \le k \le 10i$ numbers work. Note that, when $i = 0$ or $j = 0$, we have a walk of length $1, 2$ respectively, so such cases are included in $P_2$. Therefore, by \eqref{eq:P1square} and \eqref{eq:P2square}, we have that 
\begin{eqnarray}\label{eq:P3square} P_{3}
&=& P_{2} + \sum_{i = 1}^{10} {10 \choose i}p^{i}(1-p)^{10 - i} \left( \sum_{k=1}^{10i} {10i \choose k}(1-p)^{10i - k}(pP_{1})^{k}\right) \nonumber\\
&=& P_{2} + \sum_{i=1}^{10}p^i(1-p)^{10-i} \left(((1-p)+pP_1)^{10i} - (1-p)^{10i}\right)\nonumber\\ 
&=& P_{2} + \sum_{i=1}^{10}p^i(1-p)^{10-i} \left(P_2^{i} - P_1^{i}\right)\nonumber\\ 
&=& P_{2} + \sum_{i=1}^{10}p^i(1-p)^{10-i}P_2^{i} -  \sum_{i=1}^{10}p^i(1-p)^{10-i}P_1^{i}\nonumber\\ 
&=& P_2+\left((1-p+pP_2)^{10} -(1-p)^{10}) - ((1-p+pP_1)^{10}-(1-p)^{10}\right)\nonumber\\ 
&=& (1-p+pP_2)^{10} \ \approx \ 8.59502\times10^{-5}. \end{eqnarray}

The next step is to compute $P_k$ for an arbitrary $k \in \mathbb{N}^+$, which can be done by induction. Suppose that, for $2 \le m \le k$, we have $P_k = (1-p+pP_{k-1})^{10}$. Similar to the idea used to compute $P_2$ and $P_3$, we have that 
\begin{eqnarray}\label{Eq: Pk+1Steps}
&& P_{k+1} \nonumber\\
&=& P_k + \sum_{a_1 = 1}^{10}{10 \choose a_1}p^{a_1}(1-p)^{10-a_1}\sum_{a_2=1}^{10a_1}{10a_1 \choose a_2}p^{a_2}(1-p)^{10a_1-a_2}\cdots \nonumber\\
&&\sum_{a_{k-1}=1}^{10a_{k-2}}{10a_{k-2} \choose a_{k-1}}p^{a_{k-1}}(1-p)^{10a_{k-2}-a_{k-1}}\sum_{a_k=1}^{10a_{k-1}}{10a_{k-1} \choose a_k}p^{a_k}(1-p)^{10a_{k-1}-a_k}p_1^{a_k}\nonumber \\
&=& P_k + \sum_{a_1 = 1}^{10}{10 \choose a_1}p^{a_1}(1-p)^{10-a_1}\sum_{a_2=1}^{10a_1}{10a_1 \choose a_2}p^{a_2}(1-p)^{10a_1-a_2}\cdots\nonumber\\ 
&& \sum_{a_{k-1}}^{10a_{k-2}}{10a_{k-2} \choose a_{k-1}}p^{a_{k-1}}(1-p)^{10a_{k-2}-a_{k-1}}\big((1-p+pP_1)^{10a_{k-1}} - (1-p)^{10a_{k-1}}\big) \nonumber\\
&=& P_k + \sum_{a_1 = 1}^{10}{10 \choose a_1}p^{a_1}(1-p)^{10-a_1}\sum_{a_2=1}^{10a_1}{10a_1 \choose a_2}p^{a_2}(1-p)^{10a_1-a_2}\cdots \nonumber \\
&&\sum_{a_{k-1}}^{10a_{k-2}}{10a_{k-2} \choose a_{k-1}}p^{a_{k-1}}(1-p)^{10a_{k-2}-a_{k-1}}\big(P_2^{a_{k-1}} - P_1^{a_{k-1}}\big). 
\end{eqnarray}
By repeating the same procedure as in calculating $P_3$, we are able to reduce the  expression \eqref{Eq: Pk+1Steps} to
\begin{eqnarray}\label{eq:sq-free_prob}
P_{k+1} &=& P_k + \sum_{a_1 = 1}^{10}{10 \choose a_1}p^{a_1}(1-p)^{10-a_1}\big(P_k^{a_1}-P_{k-1}^{a_1}\big)\nonumber\\
&=& P_k + (1-p+pP_k)^{10}- (1-p+pP_{k-1})^{10} \nonumber\\
&=& (1-p+pP_k)^{10}, \end{eqnarray}which holds true for any positive integer $k \ge 1$.

We now prove that $P_{k}$ approaches some constant as $k \to \infty$. Using $\eqref{eq:sq-free_prob}$, we have that $$P_{k} \ = \ (1-p+pP_{k-1})^{10} \ \ge \ 0.$$
Furthermore, if $P_{k - 1} \ \le \ 1/2$, then $$ P_{k} \ \le \ \left(1 - p + \frac{p}{2}\right)^{10} \ = \ \left(1 - \frac{3}{\pi^2}\right)^{10} \ < \ 0.7^{10} \ < \ \frac{1}{2}.$$ 
Then, by induction, when the base case is $P_1 \approx 8.5835\times10^{-5}$ (from \eqref{eq:P1square}), we have that $P_{k} \le 1/2$ for any $k \ge 1$. Lastly, note that $P_{2} \ > \ P_{1}$, and using strong induction and $\eqref{eq:sq-free_prob}$, we get that 
$$P_{k+1} \ = \ (1 - p + pP_{k})^{10} \ \ge \ (1 - p + pP_{k - 1})^{10} \ = \ P_{k}.$$
In other words, $\big(P_{k}\big)_{k \ge 1}$ is an increasing sequence. By the Monotone Convergence Theorem \cite[Theorem 2.4.2]{Abb}, we get that there exists $ l \in [0, 0.5]$ such that 
\begin{equation}\lim_{k \to \infty} P_{k} \ = \ l.
\end{equation}
Sending $k \to \infty$ in $\eqref{eq:sq-free_prob}$, we get that
$$l \ = \ (1 - p + pl)^{10}.$$
Using Mathematica, we see that the only rational root in the range $[0,0.5]$ is 
\be \label{eq:limitval} l \ \approx \ 8.59502\times10^{-5}. \ee
The limit $l$ stands for the probability that, starting at some fixed number $x$, there is a bounded limit $N$, which can be very large, such that no square-free walk can exceed length~$N$. That is, if the limit of $P_k$ is as small as $8.59502\times10^{-5}$, it implies the following theorem.  

\begin{thm}\label{thm:highprobSF} Given we append one digit at a time, the theoretical probability that there is an infinite random square-free walk from any starting point is as least $1-l \approx 0.99991$. In other words, there is such a walk from almost any starting point.\end{thm}

\begin{rek}
Although Theorem \ref{thm:highprobSF} suggests that stochastically, the probability of walking to infinity on square-free numbers is high, there exist square-free numbers that can't be extended. For example, $231546210170694222$ is a square-free number, such that if we append any digit to the right we get a non-square-free number. In particular, if we delete any number of digits to the right we get a square-free number as well, so this proves we can reach a stopping point when starting with $2$ and append digits to the right randomly. Furthermore, our example implies that the walk is not constructive, i.e., if we start with a square-free walk and append a digit at random that yields a new square-free number, we may reach a point where we could not move forward.
\end{rek}


\subsection{Quantitative Results}\label{results}

From \S \ref{sec:rmlesfw}, according to the blind unlimited model of square-free walks, the expected length of square-free walks is $6/{(\pi^2-6)}$ in \textit{any} base. In reality, however, this is not always the case.

Dropping the probabilistic assumption about the square-free numbers, we assume that a random square-free walk starts with the empty string, then randomly selected digits are appended to the right, and the process stops when the number obtained is not square-free. We let $E_b$ denote the theoretical expected length of such a walk in base $b$, and $\sfr$ the set of square-free numbers. We supplement this notation with another definition.

\begin{defi}[Right Truncatable Square Free]\label{def:RTSF} We denote by $\rtsf_b$ the set of square-free numbers base $b$ such that if we successively remove the rightmost digit, each resulting number is still square-free. Equivalently, let $b^{k-1} \le x < b^{k}$. Then, $x \in \rtsf_b$ if and only if for all $\ell \in \{0, 1, \dots, k-1\}$ we have $\lfloor x/b^\ell\rfloor$ is square-free. \end{defi}
We also define \be L_{b,k}\ := \ \left|\rtsf_b\cap [b^{k-1}, b^k)\right|.\ee
This quantity counts the number of right-truncatable square-free numbers with exactly $k$ digits in base $b$.

\begin{lem}\label{EbLemma} We have \be E_b \ = \ \su_{k=1}^\infty \frac{L_{b,k}}{b^k - b^{k - 1}}.\ee \end{lem}
\begin{proof} The proof follows from the same reasoning used to prove  \eqref{eq:ABexpansionprimes}. \end{proof}

\begin{thm}\label{E2estimate} We have $E_2$ satisfies the following bounds: \be 2.31435013 \ < \ \frac{636163720502}{2^{38}}\ \le\ E_2\ \le\ \frac{636163930777}{2^{38}}\ < \ 2.31435090.\ee\end{thm}

\begin{proof} A straightforward calculation yields $(L_{2,n})_{1\le n\le 40}=(1,2,3,5,7,\ldots ,168220)$.

Let \be S_{1} \ := \ \su_{i=1}^{40}\frac{L_{2,i}}{2^i} \ =\ \frac{318081860251}{2^{37}} \label{eq:L2}, \ee
\be S_2 \ := \ \sui{41} \frac{L_{2, i}}{2^i} \nonumber \ee
and note that 
\be S_1 + S_2 \ = \ E_2. \ee
Moreover, let $L_{2,k}^O \ = \ \big|\rtsf_2\cap [2^{k-1}, 2^k)\cap (2\z + 1)\big|$ be the number of odd right truncatable square-free binary numbers of length-$k$ binary numbers, and similarly $L_{2,k}^E \ = \ \big|\rtsf_2\cap [2^{k-1}, 2^k)\cap 2\z \big|$ the even ones. By modulo $4$ considerations, we have that $L_{2,k+1}^O \ \le \ L_{2,k}^O+L_{2,k}^E$ and therefore,  we have $L_{2,k+1}^E \ \le \ L_{2,k}^O$.
\vspace{0.5cm}

$\begin{aligned}S_2
&\ = \ \frac{L^O_{2,41}+L^E_{2,41}}{2^{41}}+\sui{41}\frac{L^O_{2,i+1}+L^E_{2,i+1}}{2^{i+1}}
\ \le \  \frac{L^O_{2,41}+L^E_{2,41}}{2^{41}}+\sui{41}\frac{2L^O_{2,i}+L^E_{2,i}}{2^{i+1}}  \\
&\ = \ \frac{L^O_{2,41}+L^E_{2,41}}{2^{41}}+\frac{S_2}{2}+\frac{L^O_{2,41}}{2^{42}}+\sui{41}\frac{L^O_{2,i+1}}{2^{i+2}}\\
&\ \le \ \frac{3L^O_{2,41}+2L^E_{2,41}}{2^{42}}+\frac{S_2}{2}+\sui{41}\frac{L^O_{2,i}+L^E_{2,i}}{2^{i+2}}\\
&\ \le \ \frac{5L^O_{2,40}+3L^E_{2,40}}{2^{42}}+\frac{3S_2}{4}.\\
\end{aligned}$\vspace{0.5cm}

Thus, we have \be S_2\ \le\ \frac{5L^O_{2,40}+3L^E_{2,40}}{2^{40}}\ \le \ \frac{5L_{2,40}}{2^{40}}\ = \ \frac{210275}{2^{38}}.\ee

As clearly $S_2 \ \ge \ 0$, $\su_{i=1}^{40}\frac{L_{2,i}}{2^i}\ \le \ E_2\ = \ \su_{i=1}^{40}\frac{L_{2,i}}{2^i} + S_2$. Substituting the numerical results from \eqref{eq:L2} yields the bound. \end{proof}

Although we do not use the base $b = 2$ model for square-free walks elsewhere in this paper, the proof is outlined here since it can be adapted to other bases.

\begin{thm}\label{E10estimate}$2.63297479 \ \le \ E_{10} \ \le \ 2.720303756$.\end{thm}

\begin{proof} The proof is similar to that of Theorem \ref{E2estimate}, using \be(L_{10,n})_{1\le n\le 8}\ = \ (6, 39, 251, 1601, 10143, 64166, 405938, 2568499)\ee and the inequalities 
\begin{eqnarray}
\begin{aligned}
    L^O_{10,k+1}\ &\le \ 5L^O_{10,k}+5L^E_{10,k}; \\
    L^E_{10,k+1}\ &\le \ 3L^O_{10,k}+2L^E_{10,k}.
\end{aligned}
\end{eqnarray}
\end{proof}

\subsection{Additional remarks on the behavior of square-free walks}

We first introduce some notation. Given a number $x$ and a digit $i$ in base $b$, we denote $\overline{xi} = b \cdot x + i$; in other words, we append $i$ to the right of $x$. The following are some remarks relating to the behavior of square-free walks. 

\begin{rek}
The fact that $E_{10} > 6/{(\pi^2 - 6)}$ was expected, since we know that $\overline{xi}$ is more likely to be square-free if $x$ is square-free. This is due to the fact that if $x$ is square-free, then $x \not \equiv 0$ (mod $p^2$) for every prime $p$. In particular, this implies that $[\overline{x0}, \overline{x9}]$ can be any segment of $\mathbb{Z}/p^2\mathbb{Z}$ except $[0,9]$, hence the chance that $\overline{xi} \not\equiv 0$ (mod ${p^2}$), $\forall i \in [0, 9]$ is slightly bigger. Notice that this behavior is consistent for any base $b$.
\end{rek}

\begin{rek}
A Python program\footnote{This script is available at \bburl{https://replit.com/@TudorPopi/non-square-free-greater-square-free\#main.py}} yields that when $x \in \{1, 2, \ldots, 1{,}000{,}000\}$ is square-free, the probability of $\overline{xi}$ is also square-free is around $0.5944$, and when $x \in \{1,\, 2,\, \ldots,\, 1{,}000{,}000\}$ is not square-free, the probability of $\overline{xi}$ being square-free is around $0.5669$. Note that both these values are larger than $6/\pi^2$. This is because small numbers have a larger chance of being square-free. Furthermore, when $x$ is smaller, i.e., $x \in \{1, 2, \ldots, 10^{n}\}, n < 6$, these probabilities are even larger. As $x$ increases, we expect the two probabilities to decrease, but they still have a small difference.
\end{rek}

\begin{rek}\label{remark:Exploration}
We also explore how the starting point affects the length of the walk. As in the prime walks, the experimental expected length of the walk decreases as the starting point increases since small numbers have a bigger chance of being square-free. This is shown in Table \ref{table:square-freecomparison}. Note that the expected length of around $2.71$ (when the starting point increases) is inside the interval given by Theorem \ref{E10estimate}. 
\end{rek}

\begin{rek}
We also consider the frequency of the digits added in our square-free walk and how this changes when we vary the walk's starting point. The results are shown in Table \ref{table:square-freefreq}, and we also make the following related, qualitative observations.
\begin{itemize} \label{sq-free-observations}
    \item Odd digits appear more often than even digits. This is because if $x$ is square-free, then it cannot be a multiple of $4$, hence even digits appear less.
    \item The frequencies of $2$ and $6$ are less than $0, 4$, and $8$. This is because if $x$ and $\overline{xi}$ are square-free and $i$ is even, then if $x$ is odd, by modulo $4$ considerations $i$ is $0, 4$, or $8$, and if $x$ is even, then $i$ is $2$ or $6$. However, $x$ is almost twice more likely to be odd; hence the frequency of $0, 4$, and $8$ is bigger than that of $2$ and $6$.
    \item We have that $5$ appears less often than any other odd digit. Similar to the above, $\overline{x5}$ is not square-free if $x$ ends with $2$ or $7$.
    \item We have that $9$ appears more often than any other digit. This is because if $x$ is square-free, then $x \not \equiv 0$ (mod $9$), hence $\overline{x9} \not \equiv 0$ (mod $9$).
    \item As the starting point increases, the frequencies stabilize. \label{remdig}
\end{itemize}
\end{rek}

\begin{rek}
By looking at the last digit, we can make informed decisions on what digit to append at each step to increase the chance the number is square-free using the Remark \ref{remdig}.
\end{rek}

\renewcommand{\arraystretch}{1.5}
\begin{table}[H]
\centering
\begin{tabular}{p{5pt}c|cccccc}
\multicolumn{2}{}{} & \multicolumn{6}{c}{Number of digits of starting point} \\
     & & 1    &  2   & 3    & 4  & 5  & 6   \\ \cline{2-8}
\multirow{9}{*}{\rotatebox[origin=c]{90}{Digit added}}
  & 0 & 10.1\% &  7.4\% &  7.6\% & 7.5\%  & 7.5\%  & 7.5\% \\
  & 1 & 14.0\% & 13.6\% & 13.2\% & 13.4\% & 13.4\% & 13.4\% \\
  & 2 & 8.4\%  &  5.5\% &  5.3\% & 5.3\%  & 5.3\%  & 5.3\% \\
  & 3 & 13.5\% & 13.5\% & 13.4\% & 13.4\% & 13.4\% & 13.3\% \\
  & 4 &  5.1\% &  8.1\% & 8.0\%  & 8.0\%  & 8.0\%  & 8.0\% \\
  & 5 & 12.1\% & 10.8\% & 10.9\% & 10.8\% & 10.8\% & 10.8\% \\
  & 6 &  8.3\% &  5.5\% &  5.4\% & 5.3\%  & 5.3\%  & 5.3\% \\
  & 7 & 13.4\% & 13.5\% & 13.2\% & 13.3\% & 13.3\% & 13.3\% \\
  & 8 &  4.9\% &  7.4\% &  8.0\% & 8.0\%  & 8.0\%  & 8.0\% \\
  & 9 &  9.7\% & 14.2\% & 14.5\% & 14.6\% & 14.6\% & 14.6\% \\
\end{tabular}
\caption{\label{table:square-freefreq} Comparing the frequency of the digits of blind unlimited square-free walks in base $10$.}
\end{table}


\subsection{Blind limited model}\label{subsec:refinedGreedy}

Lastly, we present an alternative to the blind unlimited square-free walk. As stated in Remark \ref{remdig}, odd digits appear most frequently. This observation inspires a different model: if we start with an odd square-free number not divisible by $5$, we can always append $0$ to get a square-free number, since the initial number is not divisible by $2$ or $5$. Then, randomly append one of $1, 3, 7$, and $9$. If the number is square-free, repeat the process, otherwise stop and record the length. Using $\eqref{eq:density}$, we get that the probability that a random odd integer, non-divisible by $5$, is square-free is 
\begin{equation*} 
p \ = \ \prod_{p \text{ prime } \neq 2, 5} \left(1 - \frac{1}{p^2}\right) \ = \ \frac{1}{\zeta(2)} \cdot \frac{1}{1 - \frac{1}{4}}\cdot \frac{1}{1 - \frac{1}{25}} \ = \ \frac{25}{3\pi^2}.
\end{equation*}

\begin{algorithm}[Blind Limited Square-Free Walk]\label{refinedGreedy}
 Start with an odd square-free number not divisible by $5$, then append $0$ to it. Choose one digit uniformly at random from the set $\{1, 3, 7, 9\}$ and append it to the right; if the resulting number is still square-free, append another $0$ and repeat the process.
\end{algorithm}

Let $X$ denote the length of the blind limited square-free walk, starting with the empty string. This is different from starting with one digit: with a $1$-digit start, the starting point is $1$, $3$ or $7$ with $1/3$ probability each, whereas this time our first append is $1$, $3$, $7$ or $9$ with $1/4$ probability each, so it has $1/4$ chance not surviving the first step due to $9$ not being square-free. 

In estimating the theoretical expected value of $Z$ we assume that the result of every appending will be square-free with probability $p$, and all the events are independent. Note that $Z$ is always even, since we append a $0$ at every second digit. Therefore, the theoretical probability is
$$\mathbb{P}[X = 2k] \ = \ p^k(1 - p) \ = \ \frac{25^{k}}{3^k\pi^{2k}} \cdot \frac{3\pi^2 - 25}{3\pi^2} \ = \ \frac{25^k(3\pi^2 - 25)}{3^{k+1}\pi^{2k + 2}}.$$
Analogously to $\eqref{eq:expected-sq}$, we have that 
$$\mathbb{E}[X] \ = \ \frac{2p}{1 - p} \ = \ \frac{50}{3\pi^2 - 25} \ \approx \ 10.84,$$
which is a lot larger than the expected walk length in the original model computed in \eqref{eq:expected-sq}. We present the experimental comparison in Table \ref{table:greedysquare-freecomparison}. The value $11.12$ is close to the theoretical 10.84, which indicates that square-free numbers have good uniformity. Observe that the earlier comment suggests that the expected length with $1$-digit starts should be $4/3$ times that with the empty-string start, and this is confirmed by the experimental values.

\begin{table}[!htbp]
\centering
 \begin{tabular}{||c c c c c c c c||}
 \hline
 Start has exactly $r$ digits & $0$ & $1$ & $2$ & $3$ & $4$ & $5$ & $ 6$\\ 
 \hline\hline
 Greedy square-free walk & 1.68 & 2.79 & 2.76 & 2.72 & 2.71 & 2.71 & 2.71 \\
 \hline
 Alternative square-free walk & 11.12 & 14.82 & 13.24 & 13.37 & 13.47 & 13.49 & 13.50 \\
 \hline
\end{tabular}
\caption{\label{table:greedysquare-freecomparison} Comparing the expected walk lengths of greedy square-free models in base $10$.}
\end{table}

\section{Conclusion}\label{sec:conclusion}
In the exploration to find a walk to infinity along some number theory sequences, given we append a bounded number of digits, we have established several results for different sequences. We chose to study prime and square-free walks in part because the primes have zero density, whereas the square-free numbers have a positive density. Where we could not prove exact results, we used stochastic models that approximated the corresponding ``true" values fairly well.

Our stochastic models motivated a conjecture that there is no walk to infinity for primes, a sequence of zero density with no discernible pattern in its occurrence, while a walk to infinity exists for square-free numbers, whose density is a positive constant. We verified this conjecture for these and other sequences, namely perfect squares and primes in smaller bases; indeed, it is impossible to walk to infinity on primes in base $2$ if appending up to $2$ digits at a time, or in bases $3$, $4$, $5$, or $6$ if appending $1$ digit at a time. Lastly, we found a way to append an even bounded number of digits indefinitely for perfect squares.

Stochastic models give us a strong inclination to determine whether we can walk to infinity along certain number theoretic sequences. The results presented in this paper suggest simple speculation that small density leads to the absence of the walks to infinity. However, as we mainly observe sequences based on their density, it remains to be determined how much other factors, such as the sequence's pattern or structure, may contribute as well. As one possible case study, one could consider walks on the Carmichael numbers, which were recently shown to have the property that the ratio of consecutive Carmichael numbers converged to $1$; see \cite{Lar}. Another option is to give more flexibility in where digits are appended; this paper only discussed \textbf{fixed-position models}, where we either append digits only to the left, or only to the right. Allowing digits to be appended to either side, or even in the middle, generates a new set of conjectures to be studied.

\textbf{Acknowledgments:} The authors would like to thank the anonymous reviewer of an early draft of this paper for their numerous and extremely valuable comments and suggestions, including the proof of Lemma \ref{lem:base3To6} parts \ref{B1} and \ref{B4}. Furthermore, we would like to thank the other Polymath REU Walking to Infinity group members in summer 2020 for their contributions to the work. The group consisted of William Ball, Corey Beck, Aneri Brahmbhatt, Alec Critten, Michael Grantham, Matthew Hurley, Jay Kim, Junyi Huang, Bencheng Li, Tian Lingyu, Adam May, Saam Rasool, Daniel Sarnecki, Jia Shengyi, Ben Sherwin, Yiting Wang, Lara Wingard, Chen Xuqing, and Zheng Yuxi. This work was partially supported by NSF grant DMS1561945, Carnegie Mellon University, and Williams College.

\begin{appendices}

\section{Impossibility of Walks}\label{subsec:impossibility}

We demonstrate that it is impossible to walk to infinity in bases $3, 4, 5$, and $6$.

\begin{lem}\label{lem:base3To6}
    The following statements hold.
    \begin{enumerate}[label=\textbf{B\arabic*}]
        \item\label{B1} It is impossible to walk to infinity on primes in base $3$ by appending a single digit at a time to the right. 
        \item\label{B2} It is impossible to walk to infinity on primes in base $4$ by appending a single digit at a time to the right. 
        \item\label{B3} It is impossible to walk to infinity on primes in base $5$ by appending a single digit at a time to the right. 
        \item\label{B4} It is impossible to walk to infinity on primes in base $6$ by appending a single digit at a time to the right. 
    \end{enumerate}
\end{lem}

\begin{proof}
\textbf{Proof of \ref{B1}:} First, note that we can only append a $2$ in base $3$, as appending a $0$ would yield a number divisible by $3$, while appending a $1$ would yield an even number. Therefore, at each step we can only append a $2$. Let $p_1, p_2, \ldots$ be the sequence formed by appending $2$ at each step. We have that
\begin{eqnarray*}
    p_1&=&p_1, \\
    p_2 &=& 3p_1+2, \\ 
    p_3 &=& 9p_1+8, \\
    &\vdots& \\
    p_i &=& 3^{i-1}p_1+3^{i-1}-1, \\
    &\vdots&
\end{eqnarray*}
Therefore, we have that 
\begin{equation}\label{eqbasethreefermat}
    p_{p_1} \equiv  3^{p_1-1}p_1+3^{p_1-1}-1 \equiv 0 \pmod{p_1},
\end{equation}
by Fermat's little theorem. Hence $p_{p_1}$ is composite, and it is impossible to walk to infinity on primes in base $3$ by appending just one digit at a time.

\vspace{0.3 cm}

\textbf{Proof of \ref{B2}:} We confine ourselves to considering only odd digits. Since $4\equiv1$ (mod $3$), appending $1$ to a prime $p\equiv2$ (mod $3$) gives $4p + 1\equiv 0$ (mod $3$), a composite. One can thus append $1$ at most a single time in walking to infinity, and so it suffices to consider the infinite subsequence over which only $3$'s are appended. Denote the elements of this subsequence as $p_1, p_2, \ldots$. Then, in a similar fashion to the extended Cunningham chains, these elements take the form
\begin{eqnarray*}
    p_1 &=& p_1, \\
    p_2 &=& 4p_1+3, \\ 
    p_3 &=& 16p_1+15, \\
    &\vdots& \\
    p_i &=& 4^{i-1}p_1+4^{i-1}-1, \\
    &\vdots&
\end{eqnarray*}
But then
\begin{equation}\label{eqbasefourfermat}
    p_{p_1} \equiv  4^{p_1-1}p_1+4^{p_1-1}-1 \equiv 0 \pmod{p_1},
\end{equation}
by Fermat's little theorem. Hence $p_{p_1}$ is composite, and thus it is impossible to walk to infinity on primes in base $4$ by appending just one digit at a time.

We shall now apply a similar argument to base $5$.

\vspace{0.3 cm}

\textbf{Proof of \ref{B3}:} In base $5$, parity mandates that we append either $2$ or $4$ at each step. If we have a prime $p \equiv 1$ (mod $3$), then $5p+4 \equiv 0$ (mod $3$), and so we must append a $2$. Moreover, if $p \equiv 1$ (mod $3$) then $5p+2 \equiv 1$ (mod $3$) as well, so we must append \textit{another} $2$, and so on until infinity. 

If $p_1 \equiv 1$ (mod $3$), then we have that
\begin{equation}\label{eqbasefivemod1prime}
p_i \ = \ 5^{i-1}p_1+\frac{5^{i-1}-1}{2}.
\end{equation}Then it is the case that $2p_i\equiv5^{i-1}-1$ (mod ${p_1}$), and so $2p_{p_1}$ is divisible by $p_1$ according to Fermat's little theorem. Therefore $p_{p_1}$ is composite.

On the other hand, if we have a $p\equiv2$ (mod $3$), then $5p+2 \equiv 0$ (mod $3$), so we must append a $4$. But $5p+4 \equiv 2$ (mod $3$) when $p \equiv 2$ (mod $3$), thus requiring that we append $4$'s unto infinity.

Given $p_1\equiv2$ (mod $3$), we find that
\begin{equation}\label{eqbasefivemod2prime}
    p_i \ = \ 5^{i-1}p_1+5^{i-1}-1.
\end{equation} Fermat's little theorem, therefore, allows us to conclude that $p_{p_1}$ is composite. The exception is when $p_1=5$, in which case we write $p_{i}=5^{i-1}p_1+5^{i-1}-1=5^{i-2}(5p_1+4)+5^{i-2}-1=5^{i-2}p_2+5^{i-2}-1$, and observe that $p_{p_2+1}$ is divisible by $p_2$.

Hence, regardless of our initial prime, there must be a composite element in the sequence produced by appending one digit to the right.

Lastly, the same argument can be used for base $6$.

\vspace{0.3 cm}

\textbf{Proof of \ref{B4}:} By parity and modulo $3$ considerations, we can only append $1$ or $5$ at each step. However, note that $1$ can be appended at most $3$ times, as $6p + 1 \equiv p + 1$ (mod $5$). Assume that we have reached a point where we can only append $5$ and let $p_1$ be this prime. Let $p_1, p_2, \ldots$ be the sequence formed by appending $5$ at each step. We have that
\begin{eqnarray*}
    p_1 &=& p_1, \\
    p_2 &=& 6p_1+5, \\ 
    p_3 &=& 6p_1+35, \\
    &\vdots& \\
    p_i &=& 6^{i-1}p_1+6^{i-1}-1, \\
    &\vdots&
\end{eqnarray*}
By Fermat's little theorem, we have that $2^{p_1 - 1} \equiv 1$ (mod $p_1$) and $3^{p_1 - 1} \equiv 1$ (mod $p_1$) hence $6^{p_1 - 1} \equiv 1$ (mod $p_1$). Therefore, we have that
\begin{equation}\label{eqbasesixfermat}
    p_{p_1} \equiv  6^{p_1-1}p_1+6^{p_1-1}-1 \equiv 0 \pmod{p_1},
\end{equation}
and so $p_{p_1}$ is composite. Therefore, it is impossible to walk to infinity on primes in base $6$ by appending just one digit at a time.

\end{proof}


\subsection{Starting with 2 (mod 3)}\label{2mod3}

In this subsection, we compare our models with the primes when our starting number is $2$ (mod $3$). The motivation is that we can only append 3 or 9 to such a prime while hoping to remain prime; any other digit would lead to a composite number divisible by 2, 3, or 5. Therefore, we refine our model to only append $3$ or $9$. In this case, the walks are shorter, but the model predictions are closer to the primes. Note that the longest prime walk with starting point $2$ (mod $3$) less than 1{,}000{,}000 has length 10, and is $$\begin{gathered}\{809243,\ 8092439,\ 80924399,\ 809243993,\ 8092439939,\ 80924399393,\ 809243993933,\\8092439939333,\ 80924399393333,\ 809243993933339\}.\end{gathered}$$
Since there are now only two possible digits to append, instead of the four that appeared in equations \eqref{eq:theoreticalModel1Simplified} and \eqref{eq:EV-refined}, the theoretical expected length of the walk is given by
\be\label{eq:modelprime} \frac{9s}{10^s}\left(\sum_{r=1}^{s}\frac{10^{r-1}}{r}\left(\sum_{n=0}^{\infty}\prod_{k = r}^{n-1} \left(1 - \left(1 - \frac{10}{2k\log{10}}\right)^{2}\right)\right)\right).\ee

\begin{table}[ht]
\centering
 \begin{tabular}{||c c c c c c c||}
 \hline
 Start has exactly $r$ digits & $1$ & $2$ & $3$ & $4$ & $5$ & $6$\\
 \hline\hline
 Greedy model & 3.34 & 1.95 & 1.64 & 1.45 & 1.34 & 1.28 \\
 \hline
 Careful greedy model & 5.25 & 3.22 & 2.43 & 2.04 & 1.77 & 1.62\\
 \hline
 Primes & 8.00 & 3.81 & 2.64 & 2.12 & 1.81 & 1.64 \\ 
 \hline
\end{tabular}
\caption{\label{table:primecomparison2mod3} Expected length of the prime walks with starting point $2$ (mod $3$); note the empty string would not be a meaningful start.}
\end{table}

We compare our model \eqref{eq:modelprime} to the primes in Table \ref{table:primecomparison2mod3}. The careful greedy model approximates the real world extraordinarily well, especially as the initial number increases. This is due to the sparsity of the primes, as usually at most one of $\{1, 3, 7, 9\}$ can be appended as the number increases.


\subsection{Walking on Mersenne Primes}\label{subsec: Mersenne}

In this subsection, we study walks on Mersenne Primes. Recall the following definition:

\begin{defi}
    A \textbf{Mersenne prime} is a prime of the form $M_n = 2^n - 1, n \in \mathbb{N}.$
\end{defi}

A necessary (but not sufficient) condition for $M_n$ to be prime is that $n$ is prime. We now prove the following result on Mersenne Primes:

\begin{thm}\label{Thm: MErsenneWalk}
    The only nontrivial walk on Mersenne primes is $3 \to 31,$ i.e. if $M_q = 10M_p + i,$ then $p = 2, q = 5, i = 1.$
\end{thm}

\begin{proof}
Let $M_q = 10M_p + i$ be a Mersenne prime walk, where $p$ and $q$ are prime and $i \in \{0, 1,\ldots, 9\}.$ We have that 
\be \label{Eq: mersenne1}
M_q \ = \ 10M_p + i \ \Rightarrow \ 2^q - 1 \ = \ 10(2^p - 1) + i \Rightarrow 9 - i \ = \ 5 \cdot 2^{p + 1} - 2^q
\ee
    By mod $5$ considerations $i \neq 5$, and due to parity, $i$ must be odd. We now consider the remaining possible values of $i$.
    \begin{itemize}
        \item If $i = 1,$ then $8 = 5\cdot 2^{p + 1} - 2^q$: for $q \le 5$, it is easy to check that the only pair that works is $(p, q) = (2, 5).$ If $q > 5,$ then $p = 2$ since the right hand side of \eqref{Eq: mersenne1} is divisible by $8$ but not by $16$. But then the right-hand side of \eqref{Eq: mersenne1} is negative, whereas the left-hand side is positive, which is a contradiction.

        \item If $i = 3$ or $i = 7$ then the left-hand side of \eqref{Eq: mersenne1} is divisible by $2$ but not by $4$. Since $8$ divides $5 \cdot 2^{p+1}$ (as $p \ge 2$), we must have that $q = 1,$ which is false since $q$ is prime.

        \item If $i = 5,$ the left-hand side is divisible by $4$ but not by $8.$ Since $8$ divides $5 \cdot 2^{p+1}$ (as $p \ge 2$), we must have that $q = 2.$ But then $8 = 5 \cdot 2^{p + 1},$ which obviously has no integer solutions.
    \end{itemize}
    Therefore, the only nontrivial walk on Mersenne Primes is $3 \to 31.$
\end{proof}

\begin{rek}\label{rmk: perfect}
    Another interesting future direction would be to study walks on perfect numbers.
\end{rek}

\end{appendices}


\bigskip
\hrule
\bigskip

\noindent 2000 {\it Mathematics Subject Classification}:
Primary 11A41; Secondary 11B05.

\noindent \textit{Keywords:}
walking to infinity, stochastic model, prime walk, square-free number, perfect square, appending unbounded digits.

\bigskip
\hrule
\bigskip

\noindent (Concerned with sequences
\seqnum{A000290},  
\seqnum{A005117}, and 
\seqnum{A024770}.) 

\bigskip
\hrule
\bigskip


\begin{thebibliography}{999988}


\bibitem{Abb}
S. Abbott, Understanding analysis. springer publication, 2015.

\bibitem{TP}
I.O. Angell and H.J. Godwin, 1977. On truncatable primes. \textit{Mathematics of computation}, 31(137), pp.265-267.

\bibitem{BaHaPi}
R.C. Baker, G. Harman, and J. Pintz, 2001. The difference between consecutive primes, II. \textit{Proceedings of the London Mathematical Society}, 83(3), pp.532-562.

\bibitem{GS}
E. Gethner and H.M. Stark, 1997. Periodic Gaussian moats. \textit{Experimental Mathematics}, 6(4), pp.289-292.

\bibitem{GWW}
E. Gethner, S. Wagon, and B. Wick, 1998. A stroll through the Gaussian primes. \textit{The American mathematical monthly}, 105(4), pp.327-337.

\bibitem{KL}
A. Kontorovich and J. Lagarias, 2010. Stochastic models for the $3x+1$ and $5x+1 $ problems and related problems, The ultimate challenge: the $3x+1$ problem, 131-188. \textit{Amer. Math. Soc.}, Providence, RI.

\bibitem{Lar}
D. Larsen, 2023. Bertrand's Postulate for Carmichael Numbers. ArXiV preprint arXiV:2111.06963.

\bibitem{Li}
B. Li, S.J. Miller, T. Popescu, D. Sarnecki, N. Wattanawanichkul, 2022. Modeling Random Walks to Infinity on Primes in $\mathbb{Z}[\sqrt{2}]$. \textit{Journal of integer sequences}, 25.

\bibitem{Loh}
P.R. Loh, 2007. Stepping to infinity along Gaussian primes. \textit{The American Mathematical Monthly}, 114(2), pp.142-151.

\bibitem{Mil}
S.J. Miller, F. Peng, T. Popescu, and N. Wattanawanichkul, 2022. Walking to infinity on the Fibonacci sequence. \textit{Proceedings of the $20$th International Conference
on Fibonacci Numbers and Their Applications}. Available at \url{https://www.fq.math.ca/Papers1/60-5/miller2.pdf}.

\bibitem{MT-B}
S.J. Miller and R. Takloo-Bighash, 2006. \textit{An invitation to modern number theory}. Princeton University Press.

\bibitem{MS}
H.L. Montgomery, 1999. H. Montgomery and K. Soundararajan. “Beyond pair correlation”, Paul Erdos and his mathematics. I. \textit{Bolyai Soc. Math. Stud}, 11, pp.507-514.


\end{thebibliography}
\end{document}